\documentclass{amsart}
\usepackage[english]{babel}
\usepackage{amsthm, amssymb, amsmath, mathrsfs, tikz-cd, mathtools, enumitem, indentfirst}
\usepackage[all]{xy}

\newtheorem{theorem}{Theorem}

\newtheorem{lemma}[theorem]{Lemma}
\newtheorem{proposition}[theorem]{Proposition}

\theoremstyle{definition}
\newtheorem{definition}{Definition}
\newtheorem{notation}{Notation}
\newtheorem*{acknowledgments}{Acknowledgments}

\theoremstyle{remark}

\newcommand{\Z}{\mathbb{Z}}
\newcommand{\C}{\mathbb{C}}
\newcommand{\g}{\mathfrak{g}}
\newcommand{\h}{\mathfrak{h}}
\newcommand{\m}{\mathfrak{m}}
\newcommand{\n}{\mathfrak{n}}

\title[EC structures on certain horospherical varieties]{Equivariant compactifications of a unipotent group by a smooth projective horospherical variety of Picard number one}
\author{Hyukmoon Choi}
\address{Department of Mathematical Sciences, KAIST, 291 Daehak-ro, Yuseong-gu, Deajeon, 34141, Republic of Korea
}
\email{gsm08056@kaist.ac.kr}
\date{August 25, 2026}

\begin{document}

\begin{abstract}
Cheong proved that if $S$ is a simple Lie group and $P$ is a parabolic subgroup, then $S/P$ admits a unique equivariant compactification of the unipotent radical $N$ of $P$, up to isomorphism, provided that $S/P$ is not isomorphic to a projective space. We generalize this result to a smooth projective horospherical variety $X$ with Picard number 1. Let $G = \mathrm{Aut}(X)$, and let $H$ be the isotropy subgroup at a point in the open $G$-orbit. We describe a unipotent subgroup $N$ of $H$, which is not necessarily the unipotent radical of $H$ when $X$ is not homogeneous. Moreover, we prove that $X$ admits a unique equivariant compactification of $N$, up to isomorphism.  
\end{abstract}

\maketitle

\section{Introduction}

Throughout, we will work over the field $\C$. An \textit{equivariant compactification} of an algebraic group $H$ is a projective $H$-variety $X$ containing $H$ as a Zariski-open orbit, where the induced action on the open orbit agrees with left multiplication. We refer to such an action as an EC-structure on the projective variety $X$. The purpose of this paper is to construct a distinguished unipotent subgroup $N$ of the automorphism group of a smooth projective nonhomogeneous horospherical variety $X$ of Picard number 1 and to prove that all EC-structures of $N$ on $X$ are isomorphic by $\mathrm{Aut}(X)$.

In \cite{HT}, Hassett and Tschinkel studied equivariant compactifications of a vector group $\C^n$ on $\mathbb{P}^n$. They established a correspondence between EC-structures of $\C^n$ on $\mathbb{P}^n$ and Artinian local algebras of length $n+1$. This correspondence shows that EC-structures of $\C^n$ on $\mathbb{P}^n$ are not unique up to isomorphism for $n \geq 2$, and that infinitely many isomorphism classes occur for $n \geq 6$.

In contrast, Fu and Hwang proved in \cite{FH} a uniqueness theorem for EC-structures of $\C^n$ on certain Fano manifolds of Picard number 1. More precisely, if $M\not\cong\mathbb{P}^n$ is a Fano manifold of Picard number 1 whose variety of minimal rational tangents is smooth, then any two EC-structures of $\C^n$ on $M$ are isomorphic.

Following the strategy of Fu and Hwang, Cheong considered equivariant compactifications of a unipotent group on a rational homogeneous space in \cite{C}. Let $\mathfrak{s} = \bigoplus_{k \in \Z} \mathfrak{s}_k$ be a simple graded Lie algebra and let $S$ be the adjoint algebraic group with Lie algebra $\mathfrak{s}$. Let $P$ be a parabolic subgroup of $S$ with Lie algebra $\mathfrak{p} = \bigoplus_{k\geq 0} \mathfrak{s}_k$, and let $P_u$ be the unipotent radical of $P$. Cheong proved that if $\mathfrak{s}$ is not isomorphic to $(A_l, \{\alpha_1\}), (C_l, \{\alpha_1\})$, then there exists a unique EC-structure of $P_u$ on $S/P$ up to isomorphism.

We extend Cheong's result from rational homogeneous spaces to smooth projective nonhomogeneous horospherical varieties of Picard number 1. Let $L$ be a connected reductive algebraic group. A normal $L$-variety is called \textit{horospherical} if there exists an open dense $L$-orbit that is isomorphic to a torus bundle over a generalized flag variety of $L$.

To motivate our construction, recall the homogeneous case $S/P$. Let $P^{-}$ be the opposite parabolic subgroup of $P$. Then $P \cap P^-$ is a Levi subgroup of $P$, and 
\[P = (P\cap P^-) \ltimes P_u\]
for the unipotent radical $P_u$ of $P$.
Consequently, 
\[P_u \cong P / (P \cap P^-)\]
as an algebraic variety.

We obtain an analogous construction for horospherical varieties. Let $X$ be a smooth projective nonhomogeneous horospherical $L$-variety of Picard number 1 and put 
\[G = \mathrm{Aut}(X).\]
In \cite{P}, Pasquier showed that the action of $L$ on $X$ has two closed orbits and one open orbit, whereas the action of $G$ has one closed orbit and one open orbit. Let $Y$ be the closed $L$-orbit contained in the open $G$-orbit. Let $o$ be a point in $Y$ and $H$ the isotropy subgroup of $G$ at $o$. Let $w_0$ be the longest element in the Weyl group of $L$ and set $H^-:= w_0Hw_0^{-1}$. We construct a unipotent subgroup $N \subset H$ such that 
\[N \cong H^- /(H \cap H^-)\]
as an algebraic variety (Lemma \ref{unipotent}). The construction is carried out in Section \ref{sec:Existence}. We note that unlike the homogeneous case, the intersection 
\[H \cap H^-\]
may not be reductive. Consequently, $N$ need not coincide with the unipotent radical of $H^-$. Nevertheless, $N$ is a subgroup of the unipotent radical $H_u^-$ of $H^-$ and
\[N \cong H_u^- /(H_u \cap H_u^-)\]
as algebraic varieties where $H_u$ is the unipotent radical of $H$ (Lemma \ref{unipotent}).

Now, we state our main result.

\begin{theorem}
\label{main}
Let $X$ be a smooth projective nonhomogeneous horospherical variety with Picard number 1. Then $X$ admits a unique EC-structure of $N$.
\end{theorem}

The existence of an EC-structure follows from the gradation of the Lie algebra of $G$. More precisely, the negative part of the gradation integrates to the unipotent subgroup $N$, and we show that the $N$-orbit through $o$ is Zariski open in $X$. To prove the uniqueness, we generalize the approach of \cite{C} using a $G$-invariant differential system and the VMRT structure on the open $G$-orbit. Using Tanaka prolongation, we identify the Lie algebra of local infinitesimal automorphisms preserving these structures with the Lie algebra of $G$. This allows us to extend an equivariant biholomorphism between two open $N$-orbits to an automorphism of $X$. This argument can also be applied to show that the blowup of $X$ along the closed $G$-orbit admits a unique EC-structure of $N$ (Theorem \ref{blowup}).

The paper is organized as follows. In Section 2, we recall the classification and basic properties of smooth projective nonhomogeneous horospherical varieties of Picard number 1. In Section 3, we construct the unipotent subgroup N and prove the existence of an EC-structure of $N$ on $X$. Section 4 reviews regular differential systems, Tanaka prolongation, and varieties of minimal rational tangents. In Section 5, we determine the infinitesimal automorphisms preserving the relevant differential system and VMRT. Finally, in Section 6, we prove the uniqueness of the EC-structure of $N$ on $X$.

\begin{acknowledgments}
The author thanks Professor Jun-Muk Hwang for suggesting the problem and for helpful guidance and comments throughout this work. The author thanks Professor Jaehyun Hong for detailed comments on an earlier draft of this paper and for valuable discussions. The author also thanks Dr. Zhijun Luo for helpful comments. This work is supported by the Institute for Basic Science (IBS-R032-D1).
\end{acknowledgments}

\section{Preliminaries}

We begin with definitions of equivariant compactifications and horospherical varieties.

\begin{definition}
Let \(H\) be an algebraic group. A projective variety \(M\) is called an \textit{equivariant compactification} of \(H\) if there is an algebraic \(H\)-action \(A: H \times M \to M\) with a Zariski-open orbit, that is equivariantly biregular to \(H\).

Given a projective variety \(M\), an algebraic action $A:H \times M \to M$ is called an EC-structure on \(M\) if $M$ is an equivariant compactification of $H$.

An isomorphism between two EC-structures \(A_i : H \times M \to M\) for \(i=1,2\) consists of an automorphism \(F : H \to H\) and a biregular morphism \(\Psi : M \to M\) such that the diagram
\[
\begin{tikzcd}
H \times M \arrow[r, "A_1"] \arrow[d, "F\times \Psi"]
  & M \arrow[d, "\Psi"] \\
H \times M \arrow[r, "A_2"]
  & M
\end{tikzcd}
\]
commutes.
\end{definition}

\begin{definition}
Let $L$ be a connected reductive algebraic group. A closed subgroup $K$ of $L$ is called \textit{horospherical} if it contains a unipotent radical of a Borel subgroup of $L$. A normal $L$-variety is called \textit{horospherical} if there exists an open dense $L$-orbit which is equivariantly isomorphic to $L/K$ for some horospherical subgroup $K$. 
\end{definition}

We recall the classification theorem for smooth projective nonhomogeneous horospherical varieties of Picard number 1.

\begin{theorem}[Theorem 0.1 and Theorem 1.11 of \cite{P}]
\label{Pasquier}
Let \(L\) be a connected reductive algebraic group. Let \(X\) be a smooth projective horospherical \(L\)-variety of Picard number 1. Assume that \(X\) is not homogeneous. Then, \(X\) is uniquely determined by its two closed \(L\)-orbits \(Y\) and \(Z\), isomorphic to \(L/P(\alpha)\) and \(L/P(\beta)\), respectively; and \((L, \alpha, \beta)\) is one of the triples of the following list.

\begin{enumerate}
\item \((B_m, \alpha_{m-1}, \alpha_m)\) with $m \geq 3$
\item \((B_3, \alpha_1, \alpha_3)\)
\item \((C_m, \alpha_i, \alpha_{i+1})\) with \(m \geq 2\) and \(i \in \{1, ..., m-1\}\)
\item \((F_4, \alpha_2,\alpha_3)\)
\item \((G_2, \alpha_2, \alpha_1)\)
\end{enumerate}
Moreover, the automorphism group of \(X\) is \((\text{SO}(2m+1) \times \C^*) \ltimes V(\omega_m), (\text{SO}(7) \times \C^*) \ltimes V(\omega_3), ((\text{Sp}(2m) \times \C^*)/\{\pm1\}) \ltimes V(\omega_1), (\text{F}_4 \times \C^*) \ltimes V(\omega_4), (\text{G}_2 \times \C^*) \ltimes V(\omega_1)\), respectively.
\end{theorem}

\begin{proposition}[Lemma 1.15 and Remark 1.16 of \cite{P}]
\label{stable}
In the notation of Theorem \ref{Pasquier}, $Z$ is stable under $\mathrm{Aut}(X)$ but $Y$ is not stable under $\mathrm{Aut}(X)$.
\end{proposition}        

Let $\g$ be the Lie algebra of $\mathrm{Aut}(X)$. According to Theorem \ref{Pasquier}, $\g$ can be written as the semidirect sum $\g = (\mathfrak{l} \oplus \C) \rhd V$ where $\mathfrak{l}$ is a simple Lie algebra and \(V\) is an irreducible $\mathfrak{l}$-module. Let $\Delta = \{\alpha_1, ..., \alpha_m\}$ be a set of simple roots of $\mathfrak{l}$. For a simple root $\alpha_i$, let $E_{\alpha_i}$ denote the element in the Cartan subalgebra of $\mathfrak{l}$ such that
\[
\alpha_j(E_{\alpha_i}) = \begin{cases} 1 & \text{if } j = i \\ 0 & \text{if } j \neq i. \end{cases}
\]

\begin{proposition}[Proposition 25 of \cite{K}]
\label{Kim}
Let \(X\) be a smooth nonhomogeneous projective horospherical variety \((L, \alpha, \beta)\) of Picard number 1. Let $Y$ denote the closed $L$-orbit isomorphic to $L/P(\alpha)$. Let $G$ be the automorphism of $X$ and let $\g = (\mathfrak{l} \oplus \C) \rhd V$ be the corresponding Lie algebra. Let \(\mathfrak{l}_k\) and \(V_k\) be eigenspaces that have eigenvalue \(k\) under the action of \(E_\alpha\). Let \(\mu\) be the largest eigenvalue of the action of \(E_\alpha\) on \(V\). Define a gradation of $\g$ by
\begin{align*}
\mathfrak{g}_{0} & = (\mathfrak{l}_{0} \oplus \C) \rhd V_{-\mu + 1}\\
\mathfrak{g}_{-p} & = \mathfrak{l}_{-p} + V_{-\mu-p+1} \; for \; p \neq 0.
\end{align*}
Let $\mathfrak{h} = \bigoplus_{k \geq 0} \g_k$ be the nonnegative part of $\g$. Then the connected Lie subgroup of $G$ with Lie algebra $\mathfrak{h}$ is the isotropy subgroup of $G$ at some point $o \in Y$.
\end{proposition}

\section{Existence}
\label{sec:Existence}

Throughout this paper, let $X$ be a smooth projective nonhomogeneous horospherical variety of Picard number 1 and let \(G=\text{Aut}(X)\) denote its automorphism group. In this section, we define a unipotent subgroup $N$ of $G$ and show that $X$ admits an EC-structure of $N$.

According to Proposition \ref{stable}, the action of $G$ on $X$ has exactly two orbits: one open and one closed. The closed $L$-orbit $Z$ is invariant under the action of $G$, while the other closed $L$-orbit $Y$ is contained in the open $G$-orbit. 

Fix a point \(o \in Y\) and let $H$ be the isotropy subgroup of $G$ at $o$. Proposition \ref{Kim} gives a gradation $\g = \bigoplus_{k \in \Z} \g_k$ such that the Lie algebra of $H$ is 
\[
\mathfrak{h} := \bigoplus_{k \geq 0} \g_k = (\bigoplus_{k\geq0} \mathfrak{l}_k \oplus \C) \rhd \bigoplus_{k = -\mu+1}^\mu V_k.
\]
Let $\n$ be the negative part $\bigoplus_{k < 0} \mathfrak{g}_k$ of $\g$. Let $N$ be the algebraic subgroup of $G$ whose Lie algebra is $\n = \bigoplus_{k < 0} \mathfrak{g}_k$. Then $N$ is a unipotent subgroup of $G$. Indeed, its Lie algebra $\n$ can be written as
\[\n = (\bigoplus_{k < 0} \mathfrak{l}_k) \oplus V_{-\mu},\]
where the two summands commute, i.e., 
\[[\bigoplus_{k < 0} \mathfrak{l}_k, V_{-\mu}] = 0.\]
Both $\bigoplus_{k < 0} \mathfrak{l}_k$ and $V_{-\mu}$ are nilpotent Lie algebras. Hence, $\n$, and therefore $N$, is unipotent.

Now, we show that $N$ is the unipotent group introduced in the Introduction. Before proceeding, we fix some notation. Let $w_0$ denote the longest element in the Weyl group of $L$. Let $H^- := w_0Hw_0^{-1}$ and let $H_u$ and $H_u^-$ be the unipotent radicals of $H$ and $H^-$, respectively. Let $\h^-$, $\h_u$, and $\h_u^-$ be the Lie algebras of $H^-$, $H_u$, and $H_u^-$, respectively.

\begin{lemma}
\label{unipotent}
In the above notation, $N$ is a subgroup of $H^-$ satisfying
\[N \cong H^- /(H \cap H^-)\]
as an algebraic variety. 
\end{lemma}

\begin{proof}
We first show
\[H_u^- = (H_u \cap H_u^-) \cdot N \text{, and } N \cong H_u^- /(H_u \cap H_u^-)\]
and then show that
\[H_u^- /(H_u \cap H_u^-) \cong H^- /(H \cap H^-)\]
as an algebraic variety.

Since the type of $L$ is one of $B_n$, $C_n$, $F_4$, and $G_2$, $w_0(\omega) = -\omega$ for any weight $\omega$ of $L$. Hence, we have
\begin{align*}
\h^- &= (\bigoplus_{k\leq0} \mathfrak{l}_k \oplus \C) \rhd \bigoplus_{k = -\mu}^{\mu-1}V_k, \\
\h \cap \h^-  &= (\mathfrak{l}_0 \oplus \C) \rhd \bigoplus_{k = -\mu+1}^{\mu-1}V_k,\\
\h_u &= \bigoplus_{k > 0} \mathfrak{l}_k \rhd \bigoplus_{k = -\mu + 1}^{\mu}V_k,\\
\h_u^- &= \bigoplus_{k < 0} \mathfrak{l}_k \rhd \bigoplus_{k = -\mu}^{\mu-1}V_k.
\end{align*}

The Lie algebra of $H_u \cap H_u^-$ is equal to 
\[\h_u \cap \h_u^- = \bigoplus_{k=-\mu+1}^{\mu-1} V_{k}.\]
On the other hand, $N$ is a nilpotent group with Lie algebra
\[\n = (\bigoplus_{k < 0} \mathfrak{l}_k) \oplus V_{-\mu}.\]
Hence, we have 
\[\h_u^- = \n + (\h_u \cap \h_u^-) \text{ and } \n = \h_u^- / (\h_u \cap \h_u^-)\]
as vector spaces, which implies that
\[H_u^- = (H_u \cap H_u^-) \cdot N \text{ and } N\cong H_u^-/(H_u \cap H_u^-)\]
as algebraic varieties.

The explicit descriptions of the Lie algebras above show that
\[H_u \cap H_u^- = (H\cap H^-) \cap H_u^-\]
and 
\[(H \cap H^-) \cdot H_u^- = H^-.\]
Hence, $H_u^-$ acts on $H^-/(H \cap H^-)$ transitively and the isotropy group at the identity coset is $H_u \cap H_u^-$. Consequently, 
\[H_u^- /(H_u \cap H_u^-) \cong H^- /(H \cap H^-)\]
as algebraic varieties.
\end{proof}

We now show the existence of an EC-structure of $N$ on $X$.

\begin{proposition}
\label{existence}
With the notation above, $X$ admits an EC-structure of $N$.
\end{proposition}

\begin{proof}
Let $H$ be the isotropy subgroup of $G$ at $o$. Then $N \cap H$ is the isotropy subgroup of $N$ at $o$. According to Proposition \ref{Kim}, the Lie algebra of $H$ is  \[\mathfrak{h} = \oplus_{k \geq 0} \mathfrak{g}_k.\]
Thus, the Lie algebra of $N \cap H$ is equal to 
\[\n \cap \mathfrak{h} = \n \cap (\oplus_{k \geq 0} \mathfrak{g}_k) = 0.\]
Since $N \cap H$ is a unipotent subgroup of $G$, $N \cap H$ is irreducible. Hence, $N \cap H$ is trivial and the orbit $N\cdot o$ is isomorphic to $N$. It follows that 
\[\mathrm{dim}(N\cdot O) =\mathrm{dim}\, N = \mathrm{dim}\, \n = \mathrm{dim}\, X.\]
Therefore, the orbit \(N\cdot o\) is a Zariski-open subset of \(X\), which is equivariantly isomorphic to \(N\).
\end{proof}

\section{Differential systems}
In this section, we discuss the general theory of differential systems and their infinitesimal automorphisms. For more details, see \cite{T} or \cite{Y}. 

\subsection{Definitions}

\begin{definition}
A {\it differential system} \(D\) on a complex manifold \(M\) is a subbundle of the tangent bundle \(TM\), and is denoted by \((M, D)\).

Let \(\mathcal{D}\) be the sheaf of sections of a differential system \(D\). Put 
\[\mathcal{D}^{-1} = \mathcal{D}, \; \mathcal{D}^{k} = \mathcal{D}^{k+1} + [\mathcal{D}, \mathcal{D}^{k+1}] \text{ for }  k < -1.\]
\(D\) is called {\it regular} if \(\mathcal{D}^k\) is locally free.
\end{definition}

For a regular differential system $(M, D)$, let $D^k$ denote the vector bundle corresponding to \(\mathcal{D}^k\).

\begin{proposition}[Proposition 1.1 of \cite{T}]
\label{}
Let $(M, D)$ be a regular differential system. Then\\
(1) There exists a unique integer $\mu > 0$ such that 
\[
\cdots =D^{-\mu -1} = D^{-\mu} \supsetneq D^{-\mu + 1} \supsetneq \cdots \supsetneq D^{-1} = D
\]
(2) $[\mathcal{D}^k, \mathcal{D}^l] \subset \mathcal{D}^{k+l}$ for all $k,l < 0$.
\end{proposition}

\begin{definition}
Let $(M, D)$ be a regular differential system with $D^{-\mu} = TM$. For each \(x \in M\), define \[\mathfrak{g}_{-1}(x) = D^{-1}(x), \; \mathfrak{g}_{k}(x) = D^k(x)/D^{k+1}(x) \; (-\mu \leq k < -1).\] 
Set 
\[\m(x) = \bigoplus_{k = -1}^{-\mu} \g_k (x).\]
The Lie bracket of local vector fields on $M$ induces a Lie bracket operation on $\m(x)$ compatible with the gradation. The resulting graded Lie algebra $\m(x)$ is called the {\it symbol algebra of $D$ at $x \in M$}.
\end{definition}

Note that the symbol algebra $\m(x) = \bigoplus_{k < 0} \g_k(x)$ of any regular differential system is nilpotent and satisfies $[\g_{k+1}(x), \g_{-1}(x)] = \g_k(x)$ for all $k < -1$. A graded Lie algebra with this property is called fundamental.

\begin{definition}
A graded Lie algebra 
\[\m = \bigoplus_{k < 0} \m_k\]
is called {\it fundamental} if $\m$ is nilpotent and \[[\m_{k+1}, \m_{-1}] = \m_k \text{ for all } k < -1.\] For a fundamental graded Lie algebra $\m$, $(M, D)$ is called {\it of type $\m$} if the symbol algebra $\m(x)$ of $D$ is isomorphic to $\m$ for all $x \in M$.
\end{definition}

\subsection{Prolongation}

Let $\m = \oplus_{k < 0} \g_k$ be a finite dimensional fundamental graded Lie algebra.

\begin{definition}
For $k < 0$, put $\g_k(\m) = \g_k$. For $k \geq 0$, define $\g_k(\m)$ by
\[
\g_k(\m) = \{u \in \bigoplus_{p < 0} (\g_{p}^*(\m) \otimes \g_{p+k}(\m)) | u([A, B]) = u(A)(B) - u(B)(A) \}
\]
Define a Lie bracket on $\g(\m) := \bigoplus_{k \in \Z} \g_k(\m)$ as follows. For $X, Y \in \m$, $[X, Y]$ is the same as the Lie bracket operation of $\m$. If $X \in \m$ and $u \in \g_k(\m)$ for some $ k \geq 0$, then $[u, X] = -[X, u] = u(X)$. If $u \in \g_k(\m)$ and $v \in \g_l(\m)$ where $k \geq l \geq 0$, then $[u, v]$ is the element of $ \oplus_{p < 0} \g_{p}^*(\m) \otimes \g_{p+k+l}(\m)$ such that $[u, v](X) = [u(X), v] + [u, v(X)]$.

A graded Lie algebra $\g(\m) = \bigoplus_{k \in \Z} \g_k(\m)$ is called the {\it prolongation of $\m$}.
\end{definition}

Let $\g_0$ be a subalgebra of $\g_0(\m)$. There is a notion of the prolongation of $(\m, \g_0)$.

\begin{definition}
For $k < 0$, put $\g_k(\m, \g_0) = \g_k$. For $k \geq 0$, define $\g_k(\m, \g_0)$ by
\[
\g_k(\m, \g_0) = \{u \in \g_k(\m) | [u, \g_{-1}] \subset \g_{k-1}(\m, \g_0)\}
\]
$\g(\m, \g_0) := \oplus_{k \in \Z} \g_k(\m, \g_0)$ is a graded Lie subalgebra of $\g(\m)$. $\g(\m, \g_0)$ is called the {\it prolongation of $(\m, \g_0)$}.
\end{definition}

The dimension of the prolongation of $(\m, \g_0)$ could be infinite.

\subsection{Infinitesimal automorphisms}
\label{sec:Infitiesimal automorphisms}

Let $\m = \bigoplus_{k< 0} \g_k$ be a fundamental graded Lie algebra and let $M(\m)$ be its adjoint group. Let $\xi$ be the Maurer-Cartan form on $M(\m)$; that is, $\xi$ is the $\m$-valued function such that 
\[\xi(A_h) = A\]
for any left invariant vector field $A$ and any $h \in M(\m)$, where $A_h$ denotes the value of $A$ at $h$. Since $\m = \bigoplus_{k < 0} \g_k$, $\xi$ decomposes as $\xi = \sum_{k < 0} \xi^k$ where $\xi^k : M(\m) \to \g_k$ denotes the $\g_k$-component of $\xi$.

\begin{lemma}[Lemma 6.2 and Lemma 6.3 of \cite{T}]
\label{Tanaka}
Let $D_\m$ denote the left-invariant differential system on $M(\m)$ defined by $D_\m = M(\m) \times \g_{-1}$. For an infinitesimal automorphism $A$ of $(M(\m), D_\m)$ defined on $U$, there exists a unique family of functions $\{f_A^k\}_{k \in \Z}$ satisfying the following conditions:
\begin{enumerate}[label = (\arabic*)]
    \item $f_A^k$ is a $\g_k(\m)$-valued function defined on $U$
    \item $f_A^k = \xi^k(A)$ for $k < 0$
    \item $df_A^k = \sum_{r < 0} [f_A^{k-r}, \xi^r]$.
\end{enumerate}
Here, $\xi$ denotes the Maurer-Cartan form on $M(\m)$.

Conversely, let $a = \sum_{k \leq l} a^k \in \bigoplus_{k \leq l}\g_k(\m)$ and $h \in M(\m)$. Then there exists a unique infinitesimal automorphism $A$ of $(M(\m), D_\m)$ defined on $M(\m)$ such that the associated family $\{f_A^k\}_{k \in \Z}$ satisfies 
\[f_A^k(h) = a^k \; (k \leq l), \quad f_A^k \equiv 0 \; (k > l).\]
\end{lemma}

Let $h \in M(\m)$ and let $U \subset M(\m)$ be a connected open subset containing $h$. Let $\mathcal{A}(U)$ be the Lie algebra of infinitesimal automorphisms of $D_\m$ defined on $U$. By Lemma \ref{Tanaka}, there exists an injective map 
\[\rho_h: \g(\m)  \xhookrightarrow{} \mathcal{A}(U).\]
However, $\rho_h$ need not be surjective. Indeed, for $A \in \mathcal{A}(U)$, the associated family $\{f^k_A\}_{k \in \Z}$ need not satisfy $f^k_A \equiv 0$ for sufficiently large $k$. On the other hand, if the dimension of $\g(\m)$ is finite, then $\rho_h$ is bijective (\cite{Y}). The following lemma shows that $\rho_h$ is a Lie algebra antihomomorphism, i.e. $\rho_h([a, b]) = -[\rho_h(a), \rho_h(b)]$.

\begin{lemma}[Lemma 6.4 of \cite{T}]
\label{Lie bracket}
Let $A, B \in \mathcal{A}(U)$. Then 
\[f_{[A, B]}^p = - \sum_{r+s = p} [f_A^r, f_B^s] \quad (p \in \Z).\]
\end{lemma}

\subsection{Variety of minimal rational tangents}

We use the same notation as in the last subsection. If $M(\m)$ admits an open embedding into a uniruled projective manifold of Picard number 1, then $M(\m)$ carries an additional structure, called the variety of minimal rational tangents (VMRT). In this case, one can consider the infinitesimal automorphisms of $(M(\m), D_\m)$ preserving the VMRT. We recall the definition of VMRT.

\begin{definition}
\label{VMRT}
Let $M$ be a uniruled projective manifold of Picard number 1. For an irreducible component $\mathcal{K}$ of the space of rational curves on $M$, let $\mathcal{K}_x$ denote the subscheme parameterizing rational curves passing through $x \in M$. An irreducible component $\mathcal{K}$ is called a {\it minimal rational component} if $\mathcal{K}_x$ is nonempty and projective for general $x \in M$.

Choose a minimal rational component $\mathscr{K}$. For a general point $x \in M$, we define a rational map $\tau_x: \mathscr{K}_x \dashrightarrow \mathbb{P}(T_xM)$ mapping a rational curve to its tangent direction. The image of $\tau_x$ is called the {\it variety of minimal rational tangents of $M$ at $x$}, or {\it VMRT of $M$ at $x$} for short.
\end{definition}

Suppose that $M$ is an open subset of a uniruled projective manifold $M$ of Picard number 1 and an equivariant compactification of $M(\m)$. We identify $M(\m)$ with the open orbit of $M$. Let 
\[\mathscr{C} \subset \mathbb{P}TM(\m)\]
be the VMRT structure induced on $M(\m)$, and assume that the action of $M(\m)$ preserves $\mathscr{C}$. This assumption is satisfied if $M$ admits a closed embedding into a projective space such that the VMRT of $M$ is determined by lines contained in $M$ and the action of $M(\m)$ on $M$ is induced by projective linear transformations of the ambient projective space. In particular, this assumption holds for the horospherical varieties considered in this paper. Let $e \in M(\m)$ the identity element, and let 
\[\widehat{\mathscr{C}_e} \subset T_eM(\m)\]
be the affine cone over the VMRT $\mathscr{C}_e$ at $e$. Let $W$ be the linear span of $\widehat{\mathscr{C}_e}$ and define 
\[G_0 = \{\varphi \in GL(W) | \varphi(\widehat{\mathscr{C}_e}) = \widehat{\mathscr{C}_e}\}.\]
Assume that $W$ contains $\g_{-1}$ and $\g_{-1}$ is invariant under $G_0$. Then there exists a Lie group homomorphism $G_0 \to GL(\g_{-1})$ and it induces a Lie algebra homomorphism $\g_0 \to \mathfrak{gl}(\g_{-1})$.

\begin{proposition}
\label{descending}
Under the notation and assumptions above, suppose moreover that the induced homomorphism $\g_0 \to \mathfrak{gl}(\g_{-1})$ is injective. Identifying $\g_0$ with its image image, assume that $\g_0 \subset \g_0(\m)$. Let $A$ be an arbitrary infinitesimal automorphism of $D_\m$ preserving $\mathscr{C}$, and let $\{f_A^k\}_{k \in \Z}$ be the family of functions given by Lemma \ref{Tanaka}. Then for all $k \in \Z$, $f_A^k$ takes values in $\g_k(\m, \g_0)$.
\end{proposition}

\begin{proof}
Let $v \in \g_{-1}$ and $\tilde{v}$ be the left-invariant vector field associated to $v$. Since $A$ preserves $D_\m$ and $\mathscr{C}$, 
\[\xi(\mathcal{L}_A(\tilde{v})) \in [g_0, v] \]
where $\xi$ denotes the Maurer-Cartan form on $M(\m)$ and $\mathcal{L}$ denotes the Lie derivative. On the other hand, Lemma \ref{Lie bracket} implies
\[\xi(\mathcal{L}_A(\tilde{v})) = \xi([A, \tilde{v}]) = -[f_A^0, v].\]
Hence, $f_A^0(x) \in g_0$ for all $x \in M(\m)$. The differential $df_A^k$ of $f_A^k$ is determined by 
\[df_A^k = \sum_{r < 0} [f_A^{k-r}, \xi^r],\] 
so the inductive argument shows that $f_A^k$ takes values in $\g_k(\m, \g_0)$.
\end{proof}

\section{Infinitesimal automorphisms on $X$}

In this section, we apply the general theory to our case. Recall that $X$ is a smooth projective nonhomogeneous horospherical variety of Picard number 1 and $G = \mathrm{Aut}(X)$.

\subsection{Regular differential system on the open orbit}

Recall that $X$ has one closed \(G\)-orbit \(Z\) and one open \(G\)-orbit $X' := X \setminus Z$. The open $G$-orbit $X'$ is naturally identified with $G/H$ where \(H\) is the isotropy subgroup at \(o \in Y\). For each $i \in \Z$, let $\mathfrak{g}^i = \oplus_{k \geq i} \mathfrak{g}_k$. Let $D_\g$ denote the $G$-invariant differential system on $X'$ defined by   
\[D_\g = G \times_H \g^{-1} / \g^0\]
where $H$ acts on $\g^{-1}/\g^0$ via $h \cdot [X] = [Ad_h(X)]$. Since \(\n = \oplus_{k < 0} \mathfrak{g}_k\) is a fundamental graded Lie algebra, we have \[D_\g^k = G \times_H \mathfrak{g}^k / \mathfrak{g}^0, \quad k < 0.\]
Hence, the symbol algebra $\m(x)$ of $D_\g$ is isomorphic to $\n = \bigoplus_{k < 0} \g_k$ for every $x \in X'$. Therefore, $(X', D_\g)$ is a regular differential of type $\n$.

By Proposition \ref{existence}, the $N$-orbit \(O := N \cdot o\) containing $o$ is Zariski open and isomorphic to \(N\). For each \(g \in G\), define $\iota_g : N \to X'$ by 
\[\iota_g(h) = ghg^{-1} \cdot (g\cdot o) = gh \cdot o.\] 
Then $\iota_g$ is an isomorphism from $N$ to $g \cdot O$. Hence, \(X'\) can be covered by Zariski open subsets \(g \cdot O\), each of which is isomorphic to \(N\).

For any $g \in G$, $\iota_g^*(D_\g|_{g \cdot O})$ is a regular differential system of type $\n$ on $N$. Moreover, under the identification of the tangent bundle $TN$ with $N \times \n$, 
\[\iota_g^*(D_\g|_{g \cdot O}) = N \times \g_{-1}\]
for any $g \in G$. We denote by $D_\n$ the differential system $D_\n = N \times \g_{-1}$.
Then 
\[D_n^k = N \times \g_k, \quad k < 0\]
and $\iota_g$ is an isomorphism from \((N, D_\n)\) to \((g \cdot O, D_\g|_{g\cdot O})\) for all $g \in G$.

\subsection{Infinitesimal automorphisms of $(X', D_\g)$}

Let $\mathcal{A}$ be the sheaf of all infinitesimal automorphisms of $(X', D_\g)$. Let $\mathscr{C}$ be the VMRT structure induced on $X'$. Note that $\mathscr{C}$ is preserved by the action of $G$. Let $\mathcal{A}_\mathscr{C}$ be the subsheaf of $\mathcal{A}$ consisting of all infinitesimal automorphisms that preserve the VMRT. In this subsection, we compute $\mathcal{A}_{\mathscr{C}}(U)$ for a connected open subset $U$ of $X'$. We begin by recalling the result on the prolongation of $\n$.

\begin{proposition}[Proposition 25 of \cite{K}]
\label{prolongation}
Let $\g$ be the Lie algebra of the automorphism group of $X$. Give a gradation to $\g = \bigoplus_{k \in \Z} \g_k$ as in Proposition \ref{Kim}. Then $\g$ is the prolongation of $(\n, \g_0)$ where $\n = \bigoplus_{k < 0} \g_k$.    
\end{proposition}

To apply Proposition \ref{descending}, we compute the linear automorphism group of $\mathscr{C}$.
Recall that we have fixed a point $o$ in the closed orbit $Y$. Let $\mathscr{C}_o \subset \mathbb{P}T_oX$ be the VMRT of $X$ at $o$, and let $\widehat{\mathscr{C}_o}$ be the affine cone over $\mathscr{C}_o$. Let $V$ be the linear span of $\widehat{\mathscr{C}_o}$ and define  
\[G(\widehat{\mathscr{C}_o}) = \{\phi \in GL(V) : \phi(\widehat{\mathscr{C}_o}) = \widehat{\mathscr{C}_o}\}.\]

\begin{proposition}[Theorem 5.4 and Lemma 8.3 of \cite{HK}, Lemma 6.2 of \cite{HL1}, and Proposition 4.11 of \cite{HL2}]
\label{g_0}
Let $\mathfrak{g}(\widehat{\mathscr{C}_o})$ be the Lie algebra of the linear automorphism group $G(\widehat{\mathscr{C}_o})$. $\g_{-1}$ is a subspace of $V$ that is invariant under $G(\widehat{\mathscr{C}_o})$. Moreover, the induced homomorphism $\mathfrak{g}(\widehat{\mathscr{C}_o}) \to \mathfrak{g}_0(\n)$ is injective and its image is $\mathfrak{g}_0$.
\end{proposition}

\begin{proof}
Suppose first that $X$ is not of type $(C_m, \alpha_i, \alpha_{i+1})$ with $i < m -1$. In this case, $V = \g_{-1}$ and the assertion follows from Theorem 5.4 and Lemma 8.3 of \cite{HK}, together with Proposition 4.11 of \cite{HL2}. Suppose now that $X$ is of the type $(C_m, \alpha_i, \alpha_{i+1})$ with $i < m -1$. Then $X$ is an odd symplectic Grassmannian, and the assertion follows from Lemma 6.2 of \cite{HL1}.
\end{proof}

Now we compte $\mathcal{A}_\mathscr{C}$.

\begin{proposition}
\label{infinitesimal automorphism}
Let $U$ be a connected open subset of $X'$. Then $\mathcal{A}_\mathscr{C}(U) = \g$.
\end{proposition}
\begin{proof}
The proof is an adaptation of Propositions 8 and 10 in \cite{C}. We first assume that \(U\) is a connected open subset of \(O\). Through $\iota_e$, we identify $\mathcal{A}_\mathscr{C}(U)$ with the Lie algebra of infinitesimal automorphisms of 
\[(\iota_e^{-1}(U), D_n|_{\iota_e^{-1}(U)})\]
preserving $\iota_e^*\mathscr{C}$.
 
Fix $h \in \iota_e^{-1}(U)$. For $A \in \mathcal{A}_\mathscr{C}(U)$, let 
\[\{f_A^k\}_{k \in \Z}\]
be the family of functions associated with $A$ by Lemma \ref{Tanaka}. Proposition \ref{descending} together with Proposition \ref{prolongation} and Proposition \ref{g_0} implies that
\[a_A := \sum_{k \in \Z} f_A^k(h)\]
is a well defined element in $\g$. By the uniqueness assertion in Lemma \ref{Tanaka}, the map
\[\mathcal{A}_\mathscr{C}(U) \to \g, \quad A \mapsto a_A\]
is injective. Consequently, 
\[\mathrm{dim}\;\mathcal{A}_\mathscr{C}(U) \leq \mathrm{dim}\; \g.\]

Conversely, $D_\g$ and $\mathscr{C}$ are invariant under the action of $G$. Hence, every $v \in \g$ determines a fundamental vector field
\[\tilde{v}(x) := \frac{d}{dt}\big|_{t=0} e^{-tv} x \]
whose restriction to $U$ belongs to $\mathcal{A}_\mathscr{C}(U)$. Since the action of $G$ is effective, the map
\[\g \to \mathcal{A}_\mathscr{C}(U), \quad v \mapsto \tilde{v}|_U\]
is injective. Therefore,
\[\mathrm{dim}\; \g \leq \mathrm{dim}\;\mathcal{A}_\mathscr{C}(U).\]
It follows that 
\[\mathcal{A}_\mathscr{C}(U) \cong \g.\]
Under the natural identification by fundamental vector fields, we write
\[\mathcal{A}_\mathscr{C}(U) = \g.\]

Now, let $U$ be an arbitrary connected open subset of $X'$. Since $O$ is dense in $X'$, $O \cap U$ is not empty. Choose a nonempty connected open subset $V$ of $O \cap U$. The restriction maps
\[\mathcal{A}_\mathscr{C}(X') \to \mathcal{A}_\mathscr{C}(U) \to \mathcal{A}_\mathscr{C}(V)\]
are injective because a holomorphic vector field on a connected complex manifold is determined by its restriction to a nonempty open subset. Since $\tilde{v} \in \mathcal{A}_\mathscr{C}(X')$ for all $v \in \g$, we have 
\[\g \subset \mathcal{A}_\mathscr{C}(X') \subset \mathcal{A}_\mathscr{C}(U) \subset \mathcal{A}_\mathscr{C}(V) = \g\] 
Therefore, 
\[\mathcal{A}_\mathscr{C}(U) = \g.\]
\end{proof}

\section{Uniqueness}
In this section, we prove the uniqueness of the EC-structure of $N$ on \(X\) by following Cheong's method in \cite{C}. More specifically, we extends a biholomorphic map between open orbits to the entire space using the infinitesimal automorphisms of a differential system and the prolongation of a graded Lie algebra.

Suppose that \(A_1 : N \times X \to X\) and \(A_2 : N \times X\to X\) are two EC-structures on \(X\). For \(i=1,2\), let \(O_i\) denote the Zariski open orbit fix a point \(o_i \in O_i\). Both \(O_1\) and \(O_2\) are contained in the open \(G\)-orbit \(X'\).

\begin{notation}
For $i=1, 2$, let $\alpha_i : N \to O_i$ be defined by $\alpha_i(h) = A_i(h, o_i)$. For each $h \in N$, let $\phi_{i, h} : X \to X$ be the map given by $\phi_{i,h} (x) = A_i(h, x)$ for $i = 1,2$. For each $h \in N$, let $\ell_h : N \to N$ denote the left multiplication by $h$.
\end{notation}

Note that for $i=1,2$, $\alpha_i$ is an isomorphism of algebraic varieties, and it satisfies 
\[\alpha_i \circ \ell_h = \phi_{i,h} \circ \alpha_i\]
for all $h \in N$.

\begin{notation}
For $i=1,2$, let $E_i$ denote the differential systems on $N$ defined by 
\[E_i := \alpha_i^*(D|_{O_i}).\]
For each $h \in N$, denote by $\m_i(h)$ the symbol algebra of $E_i$ at $h$ ($i=1,2$).
\end{notation}

Recall that $D_\g = G \times_H g^{-1}/g^0$ is a $G$-invariant differential system on $X'$. $D_\g$ is invariant under $\phi_{i,h}$, since $\phi_{i,h} \in \mathrm{Aut}(X) = G$.

\begin{lemma}
\label{differential system}
For each $i=1, 2$, $E_i$ satisfies the following properties.
\begin{enumerate}[label = (\arabic*)]
    \item \(E_i\) is invariant under the left multiplications of $N$;
    \item $E_i^k = \alpha_i^*(D_\g^k|_{O_i})$;
    \item $(N, E_i)$ is a regular differential system of type $\n$.
\end{enumerate}
\end{lemma}

\begin{proof}
(1) follows from the fact that $\alpha_i \circ \ell_h = \phi_{i, h} \circ \alpha_i$ and that $D_\g$ is invariant under $\phi_{i, h}$. (2) follows from the fact that $\alpha_i$ is an isomorphism from $(N, E_i)$ to $(O_i, D_\g|_{O_i})$. For (3), observe that $\m_i(e)$ is isomorphic to the symbol algebra $\m(o_i)$ of $D_\g$ at $o_i = \alpha_i(e)$. As discussed in Section 3.2, $\m(o_i) \cong \n$, and therefore $\m_i(e) \cong \n$.
\end{proof}

To prove the uniqueness of EC-structures, we construct a biholomorphic map between open orbits.

\begin{lemma}
\label{local isomorphism}
There exists a Lie group automorphism \(F: N \to N\) such that the biholomorphic map $\Psi: O_1 \to O_2$ defined by $\Psi = \alpha_2 \circ F \circ \alpha_1^{-1}$ satisfies 
\begin{enumerate}[label=(\arabic*)]
    \item \(\Psi(o_1) = o_2\)
    \item \(\Psi(A_1(h, x)) = A_2(F(h), \Psi(x))\) for all \(h \in N\)
    \item for each \(u \in O_1\), the differential \(d\Psi_u\) sends \((D_\g)_u\) to \((D_\g)_{\Psi(u)}\) and \(\mathscr{C}_u\) to \(\mathscr{C}_{\Psi(u)}\). 
\end{enumerate}
\end{lemma}

\begin{proof}
The proof is an adaptation of Proposition 16 in \cite{C}.

For each $i=1, 2$, Lemma \ref{differential system} implies that two graded Lie algebra $\m_i(e)$ and $\n$ are isomorphic. Hence, there exists a Lie group automorphism $\theta_i: N \to N$ such that $D_n = \theta_i^*(E_i)$.

Choose $\varphi \in G$ such that $\varphi(o_1) = o_2$. Since $D_\g$ is invariant under the action of $G$, $\varphi$ is an automorphism of $(X', D_\g)$. On the other hand, $\beta_i := \alpha_i \circ \theta_i$ is an isomorphism from $(N, D_\n)$ to $(O_i, D_\g|_{O_i})$ such that $\beta_i(e) = o_i$. It follows that $\beta_2^{-1} \circ \varphi \circ \beta_1$ defines a local automorphism of $(N, D_\n)$, i.e., an isomorphism between $(W_1, D_\n|_{W_1})$ and $(W_2, D_\n|_{W_2})$ for some open neighborhoods $W_1, W_2$ of $e$ in $N$. 

The differential $d (\beta_2^{-1})_{o_2} \circ d\varphi_{o_1} \circ d(\beta_1)_e$ induces an Lie algebra automorphism $\tilde{f}$ of the symbol algebra of $(N, D_\n)$ at $e$, which is the graded Lie algebra $\n$. Since $N$ is simply connected, $\tilde{f}$ integrates to a Lie group automorphism $\tilde{F} : N \to N$. Let $F = \theta_2 \circ \tilde{F} \circ \theta_1^{-1}$ and $f = dF_e$. Then $F:N \to N$ is a Lie group automorphism of $N$ and $f$ is a Lie algebra automorphism of $\n$. Moreover, 
\begin{align*}
(d\alpha_2)_e \circ f \circ d(\alpha_1^{-1})_{o_1} &=  (d\alpha_2)_e \circ d(\theta_2)_e \circ \tilde{f} \circ d(\theta_1^{-1})_{e} \circ d(\alpha_1^{-1})_{o_1}\\
&= (d\beta_2)_e \circ f \circ d(\beta_1^{-1})_{o_1}\\
&= d\varphi_{o_1}.
\end{align*}
We now show that $\Psi = \alpha_2 \circ F \circ \alpha_1^{-1}$ satisfies properties (1)-(3).

Property (1) is immediate since 
\[\Psi(o_1) = \alpha_2 \circ F \circ \alpha_1^{-1}(o_1) = \alpha_2\circ F(e) = \alpha_2(e) = o_2.\]

For property (2), 
\begin{align*}
\Psi(A_1(h, x)) &= \alpha_2\circ F \circ \alpha_1^{-1} \circ \phi_{1, h} (x)\\
&= \alpha_2 \circ F \circ \ell_h \circ \alpha_1^{-1}(x)\\
&= \alpha_2 \circ \ell_{F(h)} \circ F \circ \alpha_1^{-1}(x)\\
&= \phi_{2, F(h)} \circ \alpha_2 \circ F \circ \alpha_1^{-1}(x)\\
&= \phi_{2, F(h)} \circ \Psi(x)\\
&= A_2(F(h), \Psi(x)).
\end{align*}
Here, the second and fourth equalities follow from $\alpha_i \circ \ell_h = \phi_{i, h} \circ \alpha_i$, while the third equality follows from the fact that $F$ is a Lie group automorphism of $N$.

It remains to prove property (3). Let $h \in N$ such that $\alpha_1(h) = u$. Then 
\[d\Psi_u ((D_\g)_u) = d(\alpha_2)_{F(h)} \circ dF_{h} \circ d(\alpha^{-1}_1)_u ((D_\g)_u).\]
Since $D$ is invariant under $\phi_{1, h}$ and $\alpha_1 \circ \ell_h = \phi_{1,h} \circ \alpha_1$, 
\[(d\alpha_1^{-1})_u ((D_\g)_u) = d(\alpha^{-1}_1)_u \circ d(\phi_{1, h})_{o_1} ((D_\g)_{o_1}) = d(\ell_h)_e \circ d(\alpha_1^{-1})_{o_1} ((D_\g)_{o_1}).\]
The equalities $F \circ \ell_h = \ell_{F(h)} \circ F$ and $\alpha_2 \circ \ell_h = \phi_{2,h} \circ \alpha_2$ imply that 
\[d(\alpha_2)_{F(h)} \circ dF_{h} \circ d(\ell_h)_e =  d(\alpha_2)_{F(h)} \circ d(\ell_{F(h)})_e \circ dF_{e} =  d(\phi_{2,F(h)})_{o_2} \circ d(\alpha_2)_{e} \circ f.\]
In summary,
\[d\Psi_u ((D_\g)_u) = d(\phi_{2,F(h)})_{o_2} \circ d(\alpha_2)_{e} \circ f \circ d(\alpha_1^{-1})_{o_1} ((D_\g)_{o_1}).\]
Since $(d\alpha_2)_e \circ f \circ d(\alpha_1^{-1})_{o_1} = d\varphi_{o_1}$ and $D$ is invariant under $\phi_{2, F(h)}$,
\begin{align*}
d\Psi_u((D_\g)_u)&= d(\phi_{2,F(h)})_{o_2}\circ d\varphi_{o_1} ((D_\g)_{o_1})\\
&= d(\phi_{2,F(h)})_{o_2}((D_\g)_{o_2}) = (D_\g)_{\Psi(u)}.
\end{align*}
In the last equation, we use the equality 
\[\phi_{2, F(h)}(o_2) = A_2(F(h), o_2) = \alpha_2(F(h)) = \Psi(\alpha_1(h)) = \Psi(u).\]

We can apply the same argument to $\mathscr{C}$ since $\mathscr{C}$ is preserved by the action of $G$.
\end{proof}

Now, we prove the uniqueness of EC-structures of $N$ on $X$.

\begin{theorem}
\label{uniqueness}
Any two EC-structures of \(N\) on \(X\) are isomorphic.
\end{theorem}

\begin{proof}
The proof is an adaptation of Proposition 13 and Theorem 17 in \cite{C}. Lemma \ref{local isomorphism} gives a biholomorphic map 
\[\Psi : O_1 \to O_2\]
between open orbits. Since \(d\Psi\) preserves both $D_\g$ and the VMRT, $d\Psi$ induces a Lie algebra isomorphism from $\mathcal{A}_\mathscr{C}(O_1)$ to \(\mathcal{A}_\mathscr{C}(O_2)\). We also denote the induced isomorphism by $d\Psi$. Under the identification 
\[\mathcal{A}_\mathscr{C}(O_1) = \g,\quad \mathcal{A}_\mathscr{C}(O_2)= \g\]
in Proposition \ref{infinitesimal automorphism}, $d\Psi$ is an automorphism of $\g$. For each $i = 1, 2$, let $\mathfrak{h}_i$ be the subalgebra of $\mathcal{A}_\mathscr{C}(O_i)$ consisting of all vector fields vanishing at $o_i$. Since $\Psi(o_1) = o_2$, 
\[d\Psi (\mathfrak{h}_1) = \mathfrak{h}_2.\]

Let \(\tilde G\) be a universal covering group of \(G\) and let \(\tilde H_i\) be a connected subgroup of \(\tilde G\) corresponding to \(\mathfrak{h}_i\). Then \(\tilde H_i\) is the isotropy subgroup of \(\tilde G\) at \(o_i\). Hence, the automorphism $d\Psi$ of \(\mathfrak{g}\) induces a biholomorphic map 
\[\widetilde\Psi : \tilde G / \tilde H_1 \to \tilde G / \tilde H_2,\] 
which naturally extends $\Psi$.

Note that the complement of $X'$ in $X$ is the closed orbit $Z$. The codimension of $Z$ is at least 2 and $X$ is Fano. Hence, $\widetilde{\Psi}$ can be extended to an automorphism $\Phi$ of the entire space $X$.

To complete the proof, consider two morphisms  $\Phi \circ A_1$ and $A_2 \circ (F, \Phi)$ from $N \times X$ to $X$. For $(h, x) \in N \times O_1$, 
\[A_2(F(h), \Phi(x)) = A_2(F(h), \Psi(x)) = \Psi(A_1(h, x)) = \Phi(A_1(h,x)).\]
Hence, $\Phi \circ A_1$ and $A_2 \circ (F, \Phi)$ coincide on a dense open subset \(N \times O_1\). It follows that \(\Phi \circ A_1 = A_2 \circ (F, \Phi)\) on the entire space \(N \times X\).
\end{proof}

\begin{proof}[Proof of Theorem \ref{main}]
The theorem follows from Proposition \ref{existence} and Theorem \ref{uniqueness}.
\end{proof}

Since $X$ is a projective uniruled manifold of Picard number 1, the Cartan-Fubini type extension theorem can be applied to prove the uniqueness of EC-structures of $N$ on $X$. However, our method applies more generally to projective compactifications of the open $G$-orbit $X'$ to which the action of $G$ on $X'$ extends.

Natural candidates for such compactifications are blowups of $X$. Recall that $X$ has two closed $L$-orbits, $Y$ and $Z$, where $Y$ is contained in the open $G$-orbit and $Z$ is the closed $G$-orbit. Since $Z$ is stable under the action of $G$, the action of $G$ on $X$ lifts naturally to the blowup of $X$ along $Z$. In contrast, since $Y$ is not $G$-stable, the $G$-action need not lift to the blowup of $X$ along $Y$.

\begin{theorem}
\label{blowup}
Let $X$ be a smooth projective nonhomogeneous horospherical variety of Picard number 1. Let $Z$ be the unique closed $G$-orbit, and let $\tilde{X}$ be the blowup of $X$ along $Z$. Then $\tilde{X}$ admits a unique EC-structure of $N$.
\end{theorem}

\begin{proof}
The variety $\tilde{X}$ is a projective compactification of $X'$ and the action of $G$ on $X'$ extends to $\tilde{X}$. Therefore, the assertion follows from the arguments used in the proofs of Proposition \ref{existence} and Theorem \ref{uniqueness}.
\end{proof}

\end{document}